\documentclass[11pt]{amsart}
\usepackage{fancyhdr}
\usepackage{txfonts}

\usepackage{amsmath,amsthm,amssymb,hyperref, mdframed}
\usepackage{tikz}
\usepackage[all]{xy}
\usetikzlibrary{decorations.markings}
\usetikzlibrary{backgrounds,shapes}
\usepackage{subfig}

\usepackage{color}

\definecolor{aliceblue}{rgb}{0.9, 0.95, 1.0}

\newcommand{\pslc}{{\mathrm{PSL}_2 (\mathbb{C})}}
\newcommand{\pslr}{{\mathrm{PSL}_2 (\mathbb{R})}}

\newcommand{\slc}{{\mathrm{SL}_2 (\mathbb{C})}}

\newcommand\TT{{\mathcal T}}

\newcommand\T{{\mathcal T}}
\newcommand\PP{{\mathcal P}}

\newcommand\Te{Teichm\"{u}ller }

\newtheorem{theorem}{Theorem}[section]
\newtheorem{prop}[theorem]{Proposition}

\newtheorem{thm}{Theorem}[section]

\theoremstyle{definition}

\newcommand*{\bigchi}{\mbox{\large$\chi$}}
\newcommand{\cp}{\mathbb{C}\mathrm{P}^1}

\title[Monodromy representations of meromorphic projective structures]{Monodromy representations\\ of meromorphic projective structures}

\author{Subhojoy Gupta}

\address{Department of Mathematics, Indian Institute of Science,
Bangalore, India}
\email{subhojoy@iisc.ac.in}

\author{Mahan Mj}

\address{School of Mathematics, Tata Institute of Fundamental Research, Homi 
Bhabha Road, Mumbai 400005, India}

\email{mahan@math.tifr.res.in}



\begin{document}

\begin{abstract} We determine the image of the monodromy map for meromorphic projective structures with poles of orders greater than two. This  proves the analogue of a theorem of Gallo-Kapovich-Marden, and answers a question of Allegretti and Bridgeland in this case.  Our proof uses coordinates on the moduli space of framed representations arising from the work of Fock and Goncharov.
\end{abstract}

\maketitle

\section{Introduction}

For a $k$-tuple $\mathfrak{n} = (n_1,n_2, \ldots , n_k)$ where $k\geq 1$ and each $n_i\geq 3$, let $\mathcal{P}_g(\mathfrak{n})$ denote the space of marked meromorphic projective structures on a surface of genus $g$ and $k$ labelled punctures,  such that the $i$-th puncture corresponds to a pole of order $n_i$. Throughout, we shall assume that the Euler characteristic of the underlying surface $S_{g,k}$  is negative. 

The monodromy map  
\begin{equation}\label{mmap}
\Phi : \mathcal{P}_g(\mathfrak{n}) \to {\widehat{\bigchi}}_{g,k}(\mathfrak{n}) 
\end{equation}
to the \textit{moduli space of framed representations} or the decorated character variety, records 
\begin{itemize}
\item the usual monodromy of the projective structure on the punctured surface, which is a representation  $\rho: \pi_1(S_{g,k})\to \pslc$, and 
\item  a configuration of $(n_i-2)$ points on $\cp$, which are the asymptotic values of the developing map at the $i$-th puncture, for each $i=1,2,\ldots k$.
\end{itemize}
See \S2 for more details and definitions.

In Theorem 6.1 of \cite{AllBrid}, Allegretti and Bridgeland  showed that the image of the monodromy map $\Phi$ is contained in the subspace of \textit{non-degenerate} framed representations  $ {\widehat{\bigchi}}_{g,k}(\mathfrak{n})^\ast$ that forms a cluster variety -- see Proposition \ref{dyl}, and \S2.4.2 for definitions.\\

In this note we prove:

\begin{thm}\label{thm1} 
The image of the monodromy map $\Phi$ is precisely the set  $ {\widehat{\bigchi}}_{g,k}(\mathfrak{n})^\ast$ of non-degenerate framed representations. 
\end{thm}

\medskip
This answers a question of Allegretti-Bridgeland (see \S1.7.1 of  in \cite{AllBrid}) in the case that the order of every pole is greater than two. As they noted in Theorem 1.3 of their paper, the analogous result for an ``unpunctured disk" (where $g=0,k=1$) is a consequence of the work of Sibuya in \cite{Sib-book}. \\

Theorem \ref{thm1} is the analogue of the theorem of Gallo-Kapovich-Marden  in \cite{GKM} for projective structures on a \textit{closed} surface of genus at least two. For a closed surface, the monodromy group must necessarily be non-elementary and must admit a lift to $\slc$, and their work showed that any such non-elementary representation $\rho$ arises in the image of the corresponding monodromy map.

The work of Gallo-Kapovich-Marden relies on finding an appropriate  pants decomposition of the closed surface, depending on $\rho$, and gluing together the projective (Schottky) structures on the pairs of pants comprising the decomposition. 
 The idea of the proof of Theorem \ref{thm1} is to consider an \textit{ideal triangulation} of the underlying marked bordered surface, and use coordinates on the moduli space of framed representations due to the work of Fock-Goncharov in \cite{FG}.  This is possible for a non-degenerate framed representation by a result of Allegretti-Bridgeland (see Theorem \ref{ab}). These coordinates can be interpreted as determining a pleated surface in $\mathbb{H}^3$, and we can then apply  the geometric (grafting) description of a meromorphic projective structure that we developed in \cite{GM1}, extending ideas of Thurston. \\
 
This note can thus be considered as a sequel of \cite{GM1}, though we shall not need the main result of that paper. Moreover in \S2, we shall recount the salient features of our previous work, to make the present article reasonably self-contained. In forthcoming work \cite{GM3}, we shall use similar ideas to prove analogous results for meromorphic projective structures with poles of order two.  \\

\textbf{Acknowledgments.} SG  acknowledges the SERB, DST (Grant no. MT/2017/000706), the UGC Center for Advanced Studies grant, and the Infosys Foundation for their support. Research of MM partly supported by a DST JC Bose Fellowship, Matrics research project grant  MTR/2017/000005  and CEFIPRA  project No. 5801-1. MM was also partially supported by the grant 346300 for IMPAN from the Simons Foundation and the matching 2015-2019 Polish MNiSW fund.  Both authors thank Dylan Allegretti for illuminating comments on a draft version of this article, and for contributing Proposition \ref{dyl}.  The paper was completed while SG was visiting Osaka University; he is grateful for their hospitality, and thanks Shinpei Baba for his invitation. We thank the anonymous referee for a careful reading and helpful comments.

\section{Preliminaries}

In this section we provide definitions and terminology, some of which have already been introduced in \S1. We shall also state some results of Fock-Goncharov (\cite{FG}), Allegretti-Bridgeland (\cite{AllBrid}) and \cite{GM1} that we use later. The identification of the subset of non-degenerate framed representations with a cluster variety as in Proposition \ref{dyl} is a new observation, that could be of independent interest.  Throughout this section  $\mathfrak{n} = (n_1,n_2, \ldots , n_k)$ shall be a $k$-tuple of integers where $k\geq 1$ and each $n_i\geq 3$.

\subsection{Marked bordered surface and its \Te space} 

A \textit{marked bordered surface} $S_{g}(\mathfrak{n})$ is an oriented surface of genus $g$ and $k$ labelled boundary components such that the $i$-th boundary component has $(n_i-2)$ distinguished points, for each $1\leq i\leq k$. These \textit{boundary points} are also labelled, and are cyclically ordered; two successive boundary points are endpoints of a \textit{boundary arc}. 

(Such a surface is denoted by $(\mathbb{S}, \mathbb{M})$  in \cite{AllBrid} and by $\widehat{S}$ in \cite{FG}.)\\

The \textit{\Te space of crowned hyperbolic surfaces} $\T_g(\mathfrak{n})$ is the space of marked hyperbolic metrics on  $S_{g}(\mathfrak{n})$, such that each boundary arc is a bi-infinite geodesic, up to the usual equivalence (of an isometry preserving the marking). 

Here, a ``marking" refers to a choice of a homeomorphism with the surface-with-boundary $S_{g}(\mathfrak{n})$ up to a relative homotopy (that fixes the boundary components pointwise).

Recall that a \textit{hyperbolic crown} is a hyperbolic annulus that is bounded by a closed geodesic boundary on one side, and a cyclically ordered chain of bi-infinite geodesics on the other, any successive pair of which enclose a \textit{boundary cusp}. Thus,  a surface in  $\T_g(\mathfrak{n})$ can be thought of as being obtained by attaching $k$ {hyperbolic crowns} to a compact hyperbolic surface of genus $g$ and $k$ geodesic boundary components, such that the $i$-th crown has $(n_i-2)$ boundary cusps. 

See \S3.2 of \cite{GM1} for more details, and Lemma 2.16 of \cite{GupWild} for a parametrization of  $\T_g(\mathfrak{n})$.

\subsection{Meromorphic projective structures}

A \textit{marked complex projective structure}, or ``projective structure" for short,  on a marked (possibly open) surface $S$ is a maximal atlas of charts to $\cp$ such that the transition maps on overlaps of charts are restrictions of M\"{o}bius maps, that is, of elements of $\pslc = \text{Aut}(\cp)$. 

Note that any such projective structure also determines a marked complex structure on $S$, that is, the underlying surface is a point in the \Te space of $S$.

Passing to the universal cover $\widetilde{S}$, the charts above can be used to define a \textit{developing map} $f:\widetilde{S} \to \cp$ that is $\rho$-equivariant, where $\rho:\pi_1(S) \to \pslc$ is the \textit{monodromy representation} of the projective structure.

The developing map $f$ is defined up to  post-composition by a M\"{o}bius transformation $A \in \pslc$, with $\rho$ defined up to the corresponding conjugation; the equivalence class of the pair $(f,\rho) \sim (A\circ f, A\circ \rho \circ A^{-1})$ is well-defined and can be thought of as an equivalent definition of a projective structure. \\

\textit{Example.}  A hyperbolic (or Fuchsian) structure on $S$ is an example of a projective structure; the hyperbolic plane (or Poincar\'{e}) disk can be thought of as a round disk in $\cp$, and hence the image of the developing map lies in $\cp$, and moreover the monodromy is a Fuchsian subgroup of $\pslr$, that is the group of \textit{real} M\"{o}bius transformations.\\

Fix a projective structure $P_0$ on a surface $S$, with underlying Riemann surface $X$. Then for any other projective structure $P$, we can associate a holomorphic quadratic differential $q$ on $X$ by taking the Schwarzian derivative of the conformal immersion $f:\widetilde{X} \to \cp$ obtained by post-composing the charts for $P$, with those of $P_0$. 

We shall refer to $q$ as recording the ``difference " of the two projective structures, namely $P - P_0$.

Conversely, given a holomorphic quadratic differential $q$ on $X$ (underlying $P_0$),  consider the Schwarzian equation
\begin{equation}\label{eq:schw}
 u^{\prime\prime} + \frac{1}{2} q u = 0 
\end{equation}
on the universal cover $\widetilde{X}$. The ratio of a pair of linearly independent solutions then determines the developing map for a new projective structure $P$.

In fact, projective structures on a \textit{closed} Riemann surface $X$ form an affine space for the vector space of holomorphic quadratic differentials on $X$. See, for example, \S2 of \cite{Hubbard}. \\

Let $X$ be a closed Riemann surface of genus $g$, and let $\mathsf{P}$ be a set of $k$ labelled points on it.  Fix a standard (holomorphic) projective structure $P_0$ on the closed surface $X$. 
A \textit{meromorphic projective structure} on the punctured Riemann surface $X \setminus \mathsf{P}$ is a projective structure such that the difference with $P_0$  is given by a holomorphic quadratic differential on $X \setminus \mathsf{P}$ that extends to a pole of order greater than two at each $p \in \mathsf{P}$.

(Here, as in \cite{GM1}, we shall exclude poles of order two; see \cite{AllBrid} for the complications that arise when they are present.)\\

Given a meromorphic quadratic differential $\hat{q}$ on $X$ obtained for a  meromorphic projective structure on $X\setminus P$ as above, such that the $i$-th point in $\mathsf{P}$  has a pole of order $n_i \geq 3$, for each $1\leq i\leq k$, the horizontal directions of $\hat{q}$ determine $(n_i-2)$ labelled directions at the pole. Thus, when we consider a real oriented blow-up of the puncture to a boundary circle, these directions determine $(n_i-2)$ distinguished points on the circle, and by such blow-ups for each point in $\mathsf{P}$, we obtain a marked bordered surface $S_{g}(\mathfrak{n})$ as defined in \S2.1. 
Thus, it is convenient to think of a meromorphic projective structure as a geometric structure on  $S_{g}(\mathfrak{n})$.   

The space $\PP_g(\mathfrak{n})$ of meromorphic projective structures on  $S_{g}(\mathfrak{n})$ can thus be identified with the space of triples $(X,\mathsf{P}, \hat{q})$, where $\hat{q}$ is a meromorphic quadratic differential on a Riemann surface $X$ of genus $g$,  with $k$  poles at the labelled points given by $\mathsf{P}$ of orders determined by the $k$-tuple $\mathfrak{n}$.
See \S3.1 of \cite{GM1} for more details, and for a parametrization of  the space $\PP_g(\mathfrak{n})$.  \\

Finally,  from the classical work studying the asymptotic behaviour of solutions of the Schwarzian equation \eqref{eq:schw} near a pole, we can deduce that the developing map of a meromorphic projective structure has exactly $(n_i-2)$ asymptotic values at a pole of order $n_i\geq 3$. See \S4.1 of \cite{GM1}, and Corollary 3.1 of that paper,  for a more precise statement.  These asymptotic values form part of the ``decorated" monodromy of the meromorphic projective structure, as mentioned in \S1; see also \S2.4.1.

\subsection{Measured laminations and grafting} 

In \cite{GM1} we developed a geometric description of meromorphic projective structures, that we briefly recall. \\

First, a \textit{measured lamination} on a crowned hyperbolic surface in $\T_g(\mathfrak{n})$ is a closed set that is a union of disjoint complete geodesics, equipped with a transverse measure that is invariant under transverse homotopy, together with the boundary arcs (geodesics) each with infinite weight.   Such an object is in fact a topological object defined on the underlying surface $S_g(\mathfrak{n})$, since it is determined by the transverse measures it induces on finitely many curves.

See \S3.3 of \cite{GM1} for definitions of measured laminations on crowned hyperbolic surfaces, and of the space of measured laminations ${ML}_g(\mathfrak{n})$.  In this article, the measured laminations that will appear would comprise a collection of weighted, disjoint geodesic lines between boundary cusps, apart from the geodesic sides of the crown ends that each have infinite weight. We shall henceforth assume that an element of  ${ML}_g(\mathfrak{n})$ is such a measured lamination. \\

The procedure of \textit{grafting}  a crowned hyperbolic surface $X$ along a measured lamination $\lambda$ on it can be described as follows; see  \S2.2 of \cite{GM1}, for example, for more details: 

The universal cover $\widetilde{X}$ can be identified with a convex subset of the Poincar\'{e} disk, and the (Fuchsian) projective structure on $X$ has developing image $\Omega$  in a round disk in $\cp$. In particular, the  developing image of the lifts of the leaves of $\lambda$ to the universal cover $\widetilde{X}$ is a collection of circular arcs.  The process of grafting changes this developing image to a new domain $\Omega^\prime$ obtained by rotating one side of each such circular arc by an angle equal to the weight of the leaf, thus inserting a wedge (or \textit{lune}) at its place on $\cp$.  Note that to each arc  $\gamma$ of infinite weight, we attach an ``infinite lune" that can be thought of as a semi-infinite chain of copies of $\cp$ as follows: take a collection of copies of $\cp$ with an identical slit at $\gamma$, indexed by $\mathbb{N}$, and identify the right side of the slit on the $i$-th copy, with  the left side of the slit on the $(i+1)$-th copy, for each $i\in \mathbb{N}$. By projecting each copy to the original $\cp$, we obtain an (immersed) domain $\Omega^\prime$ in $\cp$ which in fact is invariant under a new M\"{o}bius group $\Gamma ^\prime < \pslc$, and the quotient $\Omega^\prime/\Gamma^\prime$ is a new projective surface that we denote by $\text{Gr}(X,\lambda)$.\\

In Proposition 4.2 of \cite{GM1} we prove:

\begin{prop}\label{prop:graft} The grafting operation on a crowned hyperbolic surface  $X\in \TT_g(\mathfrak{n})$ along a measured lamination $\lambda$ on it  results in a projective structure  $P= \text{Gr}(X,\lambda)$ that lies in $\mathcal{P}_g(\mathfrak{n})$. 
\end{prop}

The main result of \cite{GM1} is  to show that in fact, \textit{any} element of  $\mathcal{P}_g(\mathfrak{n})$ arises from such a construction. This is a generalization of Thurston's grafting theorem (see, for example, \cite{KamTan}).\\

A key geometric object associated with the inverse map, that also plays an important role in the proof of Thurston's theorem, is the  $\rho$-equivariant map 
\begin{equation}\label{pplane}
\Psi_P: \widetilde{X} \to \mathbb{H}^3
\end{equation}
that, briefly described, is the envelope of the convex hull of maximal round disks in the developing image of $P$. (Here $\rho$ is the usual holonomy of the projective structure $P$ on the underlying punctured surface.) The image of $\Psi$ is a piecewise-totally geodesic surface called a \textit{pleated plane}, obtained by starting with $\widetilde{X}$ on an equatorial disk in $\mathbb{H}^3$, and then ``bending" along a collection of disjoint geodesic lines that are the leaves of $\widetilde{\lambda}$. See Theorem 2.1 of \cite{GM1} for more on this construction.

\subsection{Moduli space of framed representations}

\subsubsection{Definitions}
Let $\mathcal{R}_{g,k} :=\text{Hom}(\pi_1(S_{g,k}), \pslc)$ be the representation variety of surface-group representations of a punctured surface $S_{g,k}$ (having genus $g$ and $k$ punctures) to $\pslc$. 

As in pg. 8 of \cite{FG}, we define the \textit{Farey set} $\mathcal{F}_\infty(g, \mathfrak{n})$ (abbreviated to $\mathcal{F}_\infty$) as follows: given a marked bordered surface $S_g(\mathfrak{n})$,  choose a hyperbolic metric on the  marked bordered surface $S_g(\mathfrak{n})$ such that the $k$ boundary components are geodesic; the universal cover can then be identified with a convex subset $\mathcal{C}$ of the Poincar\'{e} disk.
As an  abstract set, $\mathcal{F}_\infty$ is the collection of points on $\partial \mathcal{C}$ that are the lifts of the distinguished points on the $k$ boundary components. Thus, $\mathcal{F}_\infty$ comes equipped with an action of  $\pi_1(S_{g,k})$, that acts by deck-translations on $\mathcal{C}$. \\

Then, for a marked bordered surface $S_g(\mathfrak{n})$, the \textit{moduli space of framed representations} ${\widehat{\bigchi}}_{g,k}(\mathfrak{n})$, introduced by Fock and Goncharov,  is the space
\begin{equation*}
{\widehat{\bigchi}}_{g,k}(\mathfrak{n})  = \left\{ (\rho, \beta_\rho) \text{ } \ \vert\ \text{ } \rho \in \mathcal{R}_{g,k} \text{ and } \beta_\rho:\mathcal{F}_\infty \to \cp \text{ is a }\rho\text{-equivariant map} \right\}\bigg\slash\!\bigg\slash\pslc 
\end{equation*}
where the GIT-quotient  above can be thought of as the quotient by the usual action of $\pslc$ by conjugation, after removing the ``bad orbits" to ensure the quotient space is Hausdorff.   Note that if we fix a fundamental domain for the $\pi_1(S_{g,k})$-action on $\mathcal{F}_\infty$, then the map $\beta_\rho$ assigns to each lift of a boundary component of $S_g(\mathfrak{n})$ a configuration of points in $\cp$; this is the additional ``decoration" at that boundary. 

Moreover, the quotient space is in fact a moduli stack -- see Lemma 1.1 and Definition 2.1  of \cite{FG}, or \S4.1 and Lemma 9.1 of \cite{AllBrid}. Indeed, the space prior to taking the GIT-quotient, in the definition of ${\widehat{\bigchi}}_{g,k}(\mathfrak{n})$ above,  is precisely the set of isomorphism classes of ``rigidified framed local systems" described in Lemma 4.2 of \cite{AllBrid}.  For a discussion concerning other notions of such a ``decorated character variety", see \S1.6.3 of \cite{AllBrid}. \\

As described in the preceding sections, a meromorphic projective structure  $P\in \mathcal{P}_g(\mathfrak{n})$ uniquely determines a \textit{decorated monodromy} $\hat{\rho} \in {\widehat{\bigchi}}_{g,k}(\mathfrak{n})$, which defines the monodromy map $\Phi$ (see Equation \eqref{mmap}). 

\subsubsection{Non-degenerate framed representations} 

Following the definition in \S4.2 in \cite{AllBrid}, we say that a framed representation $\hat{\rho}\in {\widehat{\bigchi}}_{g,k}(\mathfrak{n})^\ast$ is \textit{degenerate} if one of the following properties are satisfied:\\
($D1$) there is some boundary arc $I$ and some lift $\tilde{I}$ in the universal cover, such that the endpoints of $\tilde{I}$ in $\mathcal{F}_\infty$ are assigned the same point in $\cp$ by $\beta_\rho$, \\
($D2$) there is a set of two points $\{p_-,p_+\} \in \cp$, such that the $\beta_\rho$-image of any point in $\mathcal{F}_\infty$ is one of them, and moreover the monodromy $\rho (\gamma)$ of any element $\gamma \in \pi_1(S_{g,k})$ preserves this pair. 

We say that a framed representation $\hat{\rho}$ is \textit{non-degenerate} if it is not degenerate. \\

\textit{Remark.} Conditions ($D1$) and ($D2$) above are equivalent to conditions (D1) and (D2) of \cite{AllBrid} respectively. Condition (D3) of \cite{AllBrid} is not relevant for this paper, since we assume all poles have order greater than two. Moreover, to translate from the language of \cite{AllBrid} to ours, note that a choice of a flat section in a neighborhood of a puncture $p$ is equivalent to choosing $\beta_\rho(\tilde{p}) \in \cp$ for a lift $\tilde{p} \in \mathcal{F}_\infty$ (the $\beta_\rho$-images of the other lifts of $p$ are determined by the $\rho$-equivariance condition).\\

As mentioned in the introduction, ${\widehat{\bigchi}}_{g,k}(\mathfrak{n})^\ast$ is the set of non-degenerate framed representations.  The work of Allegretti-Bridgeland shows that this is a Zariski-open subset of the moduli space of framed representations (see Lemma 4.5 of \cite{AllBrid}), and moreover, we have: 

\begin{thm}[Theorem 6.1 of \cite{AllBrid}]\label{abim} The image of the monodromy map $\Phi:\mathcal{P}_g(\mathfrak{n}) \to {\widehat{\bigchi}}_{g,k}(\mathfrak{n})$ lies in ${\widehat{\bigchi}}_{g,k}(\mathfrak{n})^\ast$.
\end{thm}

\subsection{Fock-Goncharov coordinates}

An \textit{ideal triangulation} of $S_{g}(\mathfrak{n})$ is a collection of interior arcs that determines a triangulation of the surface having all vertices at the distinguished points on the boundaries -- see \S9.1 of \cite{AllBrid} for an expository account.
Note that some triangles could have one or two of its edges that are boundary arcs on  $S_{g}(\mathfrak{n})$, between successive distinguished points.\\

Fix an ideal triangulation $T$ of $S_{g}(\mathfrak{n})$.
From the work of Fock-Goncharov in \cite{FG}, given a generic framed representation $\hat{\rho} \in {\widehat{\bigchi}}_{g,k}(\mathfrak{n})$, any arc $a\in T$ determines a (non-zero) complex number $C(\hat{\rho},a)$ as follows:

Passing  to the universal cover, a lift  $\widetilde{a}$ determines a quadrilateral $Q$ comprising the two triangles of $T$ that are adjacent to $\widetilde{a}$. The framed representation $\hat{\rho}$ determines a quadruple of ``flags". In our case these flags are points in $\cp$, corresponding to $Q$, and can be thought of as associated to the vertices of $Q$.  The complex number $C(\hat{\rho},a)$ is then defined to be the \textit{cross-ratio} of these four points in $\cp$.   This is well-defined, i.e.\ it does not depend on the choice of the lift of $a$, since  the cross-ratio is invariant under elements of $\pslc$.

The following is then a special case of the ``Decomposition Theorem" of Fock-Goncharov (Theorem 1.1 in \cite{FG}). See \S2 and \S3.1 of \cite{Palesi} for an expository account.

\begin{thm}[Fock-Goncharov] \label{fg} For $g, k, \mathfrak{n}$ and a triangulation $T$ of $S_g(\mathfrak{n})$ as above, there is a birational isomorphism 
\begin{equation}\label{phit}
\phi_T:{\widehat{\bigchi}}_{g,k}(\mathfrak{n}) \to (\mathbb{C}^\ast)^{\lvert T\rvert}
\end{equation}
 where the map assigns to a (generic) framed representation $\hat{\rho}$, the tuple of cross-ratios $\{C(\hat{\rho},a) \ \vert\  a \in T\}$. 
\end{thm}

We shall also need the following result of Allegretti-Bridgeland (Theorem 9.1 of \cite{AllBrid}):

\begin{thm}[Allegretti-Bridgeland]\label{ab}  For any non-degenerate framed representation $\hat{\rho} \in {\widehat{\bigchi}}_{g,k}(\mathfrak{n})^\ast$ there is an ideal triangulation $T$ such that the map $\phi_T$ in Equation \eqref{phit} is defined for $\hat{\rho}$.
\end{thm}

\medskip

We say that a framed representation $\hat\rho \in {\widehat{\bigchi}}_{g,k}(\mathfrak{n})$ is \textit{generic} with respect to an ideal triangulation $T$ of $S_g(\mathfrak{n})$, if  the map $\phi_T$ in equation \eqref{phit} is defined at $\hat\rho$. In that case, the endpoints of each edge of the triangle $T$,  including the boundary arcs of $S_g(\mathfrak{n})$, are assigned distinct points in $\cp$. Note that for a fixed $T$, the set of framed representations that are generic with respect to $T$ is a complex algebraic torus in the moduli space of framed representations. The set of framed representations in $ {\widehat{\bigchi}}_{g,k}(\mathfrak{n})$ which are generic with respect to some ideal triangulation forms the \textit{cluster variety} associated with $S_g(\mathfrak{n})$ - see \cite{FG2} or \cite{Dyl1} for precise definitions.\\

The following is an observation of Dylan Allegretti, that he communicated to us: 

\begin{prop}\label{dyl} The cluster variety associated with $S_g(\mathfrak{n})$ coincides with the space ${\widehat{\bigchi}}_{g,k}(\mathfrak{n})^\ast$  of non-degenerate framed representations.
\end{prop}

\begin{proof}
The inclusion in one direction, namely that any non-degenerate representation lies in the cluster variety, is Theorem \ref{ab} above. 
Conversely, suppose $L$ is a point in the cluster variety. Then $L$ is generic with respect to some ideal triangulation $T$. If $L$ is degenerate, then it must have one of the properties ($D1$) or ($D2$) in \S2.4.2.  If L has the property ($D1$), then there is some boundary arc $e$ such that its endpoints are assigned the same point in $\cp$, contradicting genericity. Therefore, suppose $L$ has the property ($D2$). If $\Delta$ is any triangle of $T$, then it follows that the points assigned to the three vertices belong to a fixed pair of points in $\cp$. In particular, there is some edge $e$ of $\Delta$ whose endpoints are assigned identical points in $\cp$, which is again a contradiction. Hence $L$ cannot satisfy either ($D1$) or ($D2$), that is, it is a non-degenerate framed representation.
\end{proof}

\section{Proof of Theorem \ref{thm1}}

 Fix a non-degenerate framed representation $\hat{\rho} \in {\widehat{\bigchi}}_{g,k}(\mathfrak{n})$. Our goal is to define a meromorphic projective structure $P\in \mathcal{P}_g(\mathfrak{n})$ whose (decorated) monodromy is $\hat{\rho}$.  Recall that the data of a framed representation includes a representation $\rho:\pi_1(S_{g,k}) \to \pslc$, together with a $\rho$-equivariant map $\beta_\rho:\mathcal{F}_\infty \to \cp$.

\subsection{Framed representation $\longrightarrow$ pleated plane} 
In this section, we show how $\hat{\rho}$ determines a geometric object (a pleated plane in $\mathbb{H}^3$). The construction here is adapted from the one in the proof of Theorem 6.1 in \cite{FG}. \\

In what follows, $\widetilde{S}$ shall denote the universal cover of the marked and bordered surface $S_g(\mathfrak{n})$. Note that $S_g(\mathfrak{n})$ is homotopy equivalent to the punctured surface $S_{g,k}$ and hence  $\widetilde{S}$ admits an action of  $\pi_1(S_g,k)$.

First, fix an ideal  triangulation $T$ such that the map $\phi_T$ in Equation \eqref{phit} is defined for $\hat{\rho}$ -- such a $T$ exists by Theorem \ref{ab}. 
Thus, there is an assignment of complex cross-ratios $\mathsf{C} = \mathsf{C}(\hat{\rho}, T)$ to the arcs of the triangulation $T$.  This data of cross-ratios $\mathsf{C}$ then determines a $\rho$-equivariant map (well-defined up to post-composition by $\pslc$)
\begin{equation}\label{pp}
\Psi: \widetilde{S} \to \mathbb{H}^3
\end{equation}
as follows:

The triangulation $T$ lifts to a triangulation $\widetilde{T}$ of $\widetilde{S}$. For any triangle $\Delta$ of $\widetilde{T}$, the image of $\Psi$ will be a totally-geodesic ideal triangle in $\mathbb{H}^3$; thus the image of $\Psi$ is uniquely determined if we specify the points on $\cp = \partial \mathbb{H}^3$ that the (ideal) vertices of $\Delta$ map to. 

We start with a choice of a triangle $\Delta_0$ with vertices $\{v_0,v_1,v_2\}$  in the triangulation $\widetilde{T}$ such that the flags (points in $\cp$) corresponding to the three vertices determined by $\hat{\rho}$, are distinct. This is possible due to the fact that $\hat{\rho}$ is non-degenerate and follows immediately from Condition ($D2$) in \S2.4.2 (see also Remark 4.4 (ii) of \cite{AllBrid}).  We require that $\Psi$ will map these vertices to $0,1, \infty \in \cp$. 

Let $\Delta_1$ be a triangle adjacent to $\Delta_0$  sharing a  geodesic side $\widetilde{a} \in \widetilde{T}$ with endpoints $v_0$ and $v_2$; let  $v_3$ be the remaining vertex of $\Delta_1$. Then we define the image of $v_3$ under $\Psi$ to be the point of $\cp$ such that the cross-ratio of the ordered tuple $\{v_0,v_1,v_2,v_3\}$ is  $C(\hat{\rho},a)$.

Let $\mathsf{T}$ be the  dual tree to the triangulation $\widetilde{T}$, defined by having a vertex for each triangle, and an edge for a pair of adjacent triangles. The construction above describes how the image of the triangle corresponding to a vertex of $\mathsf{T}$ determines the image of the triangle corresponding to an adjacent vertex.  Since $\widetilde{S}$ is connected, so is $\mathsf{T}$, and proceeding inductively along the edges of $\mathsf{T}$, we can determine the entire image of the map $\Psi$ (and the map itself up to an equivariant isotopy).\\

Note that the image of the map $\Psi$ is a \textit{pleated plane} in the sense of Thurston (see \cite[Chapter 8]{thurstonnotes}). 

Moreover, it follows from the construction that
\begin{itemize}
\item[(a)] the map $\Psi$ is $\rho$-equivariant, and 
\item[(b)] it determines a $\rho$-invariant map $\beta_\rho$ from the ideal vertices, i.e.\ the lifts of the distinguished points on each boundary, to $\cp$, 
\end{itemize} 
where $\rho$ and the map $\beta_\rho$ are exactly the data of the framed representation $\hat{\rho}$.  This is in fact the proof of the injectivity of the map $\phi_T$ in Theorem \ref{fg}; for details see the proof of Theorem 6.2 in \cite{FG}.

 \subsection{Pleated plane $\longrightarrow$ projective structure}
 Recall that we have fixed a non-degenerate framed representation $\hat{\rho} \in {\widehat{\bigchi}}_{g,k}(\mathfrak{n})$. The construction in the previous section determines a pleated plane in $\mathbb{H}^3$, given by Equation \eqref{pp}. We shall denote the pleated plane by $\mathcal{P}$.
 
\subsubsection{Crowned hyperbolic surface} We first describe how the pleated plane $\mathcal{P}$ determines a crowned hyperbolic surface $\widehat{S} \in  \TT_g(\mathfrak{n})$.\\

Recall that from our construction, the pleated plane $\mathcal{P}$ is pleated along the collection of geodesic lines $\widetilde{T}$ arising from the lift of the triangulation $T$ of the marked bordered surface $S_g(\mathfrak{n})$.  

At each arc  $a\in T$, the complex cross-ratio $C(\hat{\rho},a)$ can be interpreted as a complex ``shear-bend" parameter, and we can determine the universal cover of a crowned hyperbolic surface $\widehat{S}$ by ``straightening" the bends. The non-degeneracy of $\hat\rho$ (see  Condition ($D1$) in \S2.4.2 ensures that the number of boundary cusps in the $i$-th crown is exactly $(n_i-2)$, for each $1\leq i\leq k$, and hence $\widehat{S}\in \TT_g(\mathfrak{n})$.
 
Alternatively,  the Teichm\"{u}ller space of crowned hyperbolic surfaces  $\TT_g(\mathfrak{n})$ can be parametrized by shear-coordinates based on the triangulation $T$  (\textit{c.f.} Definition 6.2 (a) of \cite{FG}). See, for example, \cite{BBFS} for an account of shear-coordinates in the usual Teichm\"{u}ller spaces.  The positive real parameters given by the modulus $\lvert C(\hat{\rho},a)\rvert$ for $a\in T$ are then the shear parameters for a crowned hyperbolic surface  $\widehat{S}\in \TT_g(\mathfrak{n})$.

\subsubsection{Measured lamination} The pleated plane $\mathcal{P}$ also determines a measured geodesic lamination $\lambda$ on the crowned hyperbolic surface $\widehat{S}$, as we now describe.  

The lamination $\lambda$ is defined to be the one comprising

\begin{itemize}
\item[(i)]  finitely many disjoint geodesic lines on $\widehat{S}$ that descend from the geodesic lines of $\widetilde{T}$; these correspond to the arcs of the original triangulation $T$, and each arc $a$ is given a  weight equal to the bending angle $\beta(a) \in [0,2\pi)$ that one sees at a lift $\tilde{a}$ on the pleated plane $\mathcal{P}$, and

\item[(ii)] the geodesic sides of the crown ends of $\widehat{S}$, each with infinite weight. 

\end{itemize}

Note that since $\widehat{S}$ is a crowned hyperbolic surface in $\TT_g(\mathfrak{n})$, the measured lamination  $\lambda \in {ML}_g(\mathfrak{n})$ (\textit{c.f.} \S2.3).

\subsubsection{Concluding the proof} 
 From  Proposition \ref{prop:graft}, we know that grafting the crowned hyperbolic surface $\widehat{S}$ along ${\lambda}$ results in a meromorphic projective structure $P \in \mathcal{P}_g(\mathfrak{n})$.
 
 Recall that such a projective structure has an equivariant developing map from $\widetilde{S}$ to $\cp=\partial \mathbb{H}^3$, that determines an equivariant map $\Psi_P:\widetilde{S} \to \mathbb{H}^3$  as in Map \eqref{pplane}. For some more details, we refer to the proof of  Theorem 2.1 in \cite{GM1}. 

Moreover, note that the projective structure $P$ arises from a grafting construction. Thus, the equivariant map $\Psi_P$ is obtained by identifying $\widetilde{S}$ with a subset of the equatorial disk in $\mathbb{H}^3$, and bending equivariantly along the lifts of  the leaves of  $\lambda$. These leaves are the arcs of the triangulation $T$, and the bending angle at a  leaf $\tilde{a}$ is $\beta(a)$ (see (i) in \S3.2.2). Infinite grafting along any geodesic line that is a lift of a crown boundary does not affect its endpoints in $\cp$. This implies that the map $\beta_\rho:\mathcal{F}_\infty \to \cp$ (see \S2.4.1) is determined by the leaves of $\lambda$ of finite weight.

 By our choice of $\widehat{S}$ and grafting lamination $\lambda$ (in particular, the weights that determine the bending angles) above, it follows that the pleated plane arising from the grafting construction is precisely  $\mathcal{P}$ that we constructed earlier.   That is, up to an equivariant isotopy, the equivariant map $\Psi_P$  is exactly the map $\Psi$ as in Equation \eqref{pp}.  By Properties (a) and (b) of $\Psi$ as observed at the end of \S2.2, the monodromy of the projective structure $P$ is $\hat{\rho}$. \\

Since $\hat{\rho}$ chosen at the beginning of \S2.2 was an arbitrary non-degenerate framed representation in ${\widehat{\bigchi}}_{g,k}(\mathfrak{n})^\ast$, this completes the proof of Theorem 1.1.


\begin{thebibliography}{0}

\bibitem[AB]{AllBrid}
Dylan Allegretti and Tom Bridgeland, \emph{The {m}onodromy of meromorphic
  {p}rojective {s}tructures}, {\it preprint, arXiv:1802.02505}.

\bibitem[All16]{Dyl1}
Dylan Allegretti, \emph{The {G}eometry of {C}luster {V}arieties from
  {S}urfaces}, ProQuest LLC, Ann Arbor, MI, 2016, Thesis (Ph.D.)--Yale
  University.

\bibitem[BBFS13]{BBFS}
M.~Bestvina, K.~Bromberg, K.~Fujiwara, and J.~Souto, \emph{Shearing coordinates
  and convexity of length functions on {T}eichm\"{u}ller space}, Amer. J. Math.
  \textbf{135} (2013), no.~6, 1449--1476.

\bibitem[FG06]{FG}
Vladimir Fock and Alexander Goncharov, \emph{Moduli spaces of local systems and
  higher {T}eichm\"{u}ller theory}, Publ. Math. Inst. Hautes \'{E}tudes Sci.
  (2006), no.~103, 1--211.

\bibitem[FG09]{FG2}
V.~V. Fock and A.~B. Goncharov, \emph{The quantum dilogarithm and
  representations of quantum cluster varieties}, Invent. Math. \textbf{175}
  (2009), no.~2, 223--286.

\bibitem[GKM00]{GKM}
Daniel Gallo, Michael Kapovich, and Albert Marden, \emph{The monodromy groups
  of {S}chwarzian equations on closed {R}iemann surfaces}, Ann. of Math. (2)
  \textbf{151} (2000), no.~2, 625--704.

\bibitem[GM]{GM1}
Subhojoy Gupta and Mahan Mj, \emph{Meromorphic projective structures, grafting
  and the monodromy map}, {\it preprint, arXiv:1904.03804}.
  


\bibitem[Gup]{GupWild}
Subhojoy Gupta, \emph{Harmonic maps and wild {T}eichm\"{u}ller spaces}, {\it to
  appear in Journal of Topology and Analysis, arXiv:1708.04780}.

\bibitem[Gup2]{GM3}
Subhojoy Gupta, \emph{Monodromy groups of $\cp$-structures on punctured surfaces}, {\it in preparation}.

\bibitem[Hub81]{Hubbard}
John~H. Hubbard, \emph{The monodromy of projective structures}, Riemann
  surfaces and related topics: {P}roceedings of the 1978 {S}tony {B}rook
  {C}onference ({S}tate {U}niv. {N}ew {Y}ork, {S}tony {B}rook, {N}.{Y}., 1978),
  Ann. of Math. Stud., vol.~97, Princeton Univ. Press, Princeton, N.J., 1981,
  pp.~257--275.

\bibitem[KT92]{KamTan}
Yoshinobu Kamishima and Ser~P. Tan, \emph{Deformation spaces on geometric
  structures}, Aspects of low-dimensional manifolds, Adv. Stud. Pure Math.,
  vol.~20, Kinokuniya, Tokyo, 1992, pp.~263--299.

\bibitem[Pal13]{Palesi}
F.~Palesi, \emph{Introduction to positive representations and
  {F}ock-{G}oncharov coordinates}, {\it Lecture notes,
  https://hal.archives-ouvertes.fr/hal-01218570}, 2013.

\bibitem[Sib75]{Sib-book}
Yasutaka Sibuya, \emph{Global theory of a second order linear ordinary
  differential equation with a polynomial coefficient}, North-Holland
  Publishing Co., Amsterdam-Oxford; American Elsevier Publishing Co., Inc., New
  York, 1975, North-Holland Mathematics Studies, Vol. 18.

\bibitem[Thu80]{thurstonnotes}
W.~P. Thurston, \emph{The {G}eometry and {T}opology of 3-{M}anifolds},
  Princeton University Notes (1980).

\end{thebibliography}
\end{document}